\documentclass[a4paper,10pt]{amsart}
\usepackage{amssymb,amsmath,graphicx}

\newtheorem*{thm}{Theorem}

\newtheorem*{lem}{Lemma}
\newtheorem*{pro}{Properties}
\newtheorem*{fact}{Fact}
\newtheorem{cor}{Corollary}

\newtheorem*{defi}{Definition}

\newcommand{\grad}{\,{\rm grad}}

\newcommand{\jj}{\,{\rm j}}

\newcommand{\Frac}{\,{\rm Frac}}
\newcommand{\gr}{\,{\rm gr}}

\newcommand{\Z}{\mathbb{Z}}

\newcommand{\N}{\mathbb{N}}
\title{A  parachute for the degree of a polynomial in algebraically independent ones}
\author{S. V\'en\'ereau}

\begin{document}

\maketitle

\begin{abstract}
  We give a  simpler proof as well as a generalization of the main result of an article of Shestakov and Umirbaev
  (\cite{SU}). This latter article being the first of two that solve a long-standing conjecture about the non-tameness,
  or "wildness",
of Nagata's automorphism. As corollaries we get interesting
informations about the leading terms of polynomials forming an
automorphism of $K[x_1,\cdots,x_n]$ and reprove the tameness of
automorphisms of $K[x_1,x_2]$.
\end{abstract}
The following notations are fixed throughout the article:  $K$ is a field of cha\-rac\-teristic 0 and $K[x_1,\cdots,x_n]$ is the ring of polynomials
 in the $n$ indeterminates $x_1,\cdots,x_n$ with coefficients in $K$,
endowed with the classical degree function: $\deg$. We consider $m$
algebraically independent polynomials in $K[x_1,\cdots,x_n]$: $f_1,
\cdots,f_m$ of respective degrees $d_1,\cdots,d_m$. There is also,
for every polynomial $G\in K[f_1,\cdots,f_m]$ a unique one $\mathcal
G(X_1,\cdots,X_m)\in K[X_1,\cdots,X_m]$, where $X_1,\cdots,X_m$  are
new indeterminates, such that $G=\mathcal G(f_1,\cdots,f_m)$. By
abuse of notation we will write $\frac{\partial G}{\partial f_i}$ to
denote $\frac{\partial \mathcal G}{\partial X_i}(f_1,\cdots,f_m)$,
$\forall 1\leq i\leq m$ and $\deg_{f_i}G$ to denote
$\deg_{X_i}\mathcal G$, the degree of $\mathcal G$ in $X_i$.

The following definition and properties are only   formally new, and come from \cite{SU}.
\begin{defi}
  We call the {\em parachute} of $  f_1, \cdots,f_m$ and denote $\nabla=\nabla(f_1,\cdots,f_m)$ the integer
  $$
  \nabla=d_1+\cdots+d_m-m-\max_{1\leq i_1,\cdots,i_m\leq n}\deg \jj_{x_{i_1},\cdots,x_{i_m}}(f_1,\cdots,f_m)
  $$
  where $\jj_{x_{i_1},\cdots,x_{i_m}}(f_1,\cdots,f_m)$ is the jacobian determinant of $f_1,\cdots,f_m$ with respect
  to $x_{i_1},\cdots,x_{i_m}$ that is $\jj_{x_{i_1},\cdots,x_{i_m}}(f_1,\cdots,f_m)=\det (\partial f_i/\partial
  x_{i_j})_{i,j}$.
\end{defi}

\begin{pro}
  The parachute of $f_1,\cdots,f_m$ has the following estimate:
  \begin{eqnarray}\label{est}&
  0\leq  \nabla=\nabla(f_1,\cdots,f_m)  \leq d_1+\cdots+ d_m-m\, .
  &\end{eqnarray}
  For any $G\in K[f_1,\cdots,f_m]$ and $\forall 1\leq i\leq m$, one
  has
  \begin{eqnarray}\label{para}&
    \begin{array}{rll}
      \deg G\geq  & \deg\frac{\partial G}{\partial f_i}+d_i-\nabla & \mbox{ and, inductively,}\\
      \deg G\geq  & \deg\frac{\partial^k G}{\partial f_i^k}+kd_i-k\nabla,& \forall k\geq 0.
    \end{array}
  &\end{eqnarray}
\end{pro}

\begin{proof}
  The left minoration  $0\leq\nabla$ in (\ref{est}) is an easy exercise
  . The right majoration is a direct
  consequence of the following
  \begin{fact}
    The vectors $\grad f_1=(\partial f_1/\partial x_1,\cdots,\partial f_1/\partial x_n), \cdots,\grad f_m=
    (\partial f_m/\partial x_1,\cdots,\partial f_m/\partial x_n)$ are linearly independent over $K[x_1,\cdots,x_n]$
    therefore the minors of order $m$ of the matrix $\left(\begin{array}{c}\grad f_1 \\ \vdots\\ \grad
    f_m\end{array}\right)$ are not all 0 and the number $\max_{1\leq i_1,\cdots,i_m\leq n}\deg
    \jj_{x_{i_1},\cdots,x_{i_m}}(f_1,\cdots,f_m)$ is non-negative.
  \end{fact}
  \begin{proof}
  As mentioned in \cite{SU}, it is a well-known fact (see e.g. \cite{ML} or \cite{Yu} for a nice proof) that $n$ rational functions $f_1,\cdots,f_n\in K(x_1,\cdots,x_n)$ are algebraically independant if and only if their jacobian determinant is not zero. Our fact is then proved by completing our $m$ algebraically independant polynomials to get $n$ algebraically independant rational functions: the jacobian determinant is not zero and it follows that $\grad f_1,\cdots,\grad f_m$ must be linearly independant.
\end{proof}
 It is clearly sufficient to show (\ref{para})  for $i=m$. \\
 Take any $m$ integers $1\leq i_1,\cdots,i_m\leq n$. From the definition of $\jj_{x_{i_1},\cdots,x_{i_m}}$ it is
 clear that
 \begin{eqnarray*}
   \deg \jj_{x_{i_1},\cdots,x_{i_m}}(f_1,\cdots,f_{m-1},G) & \leq & d_1-1+\cdots+d_{m-1}-1+\deg G-1\\
                                                           & \leq& d_1+\cdots+d_{m-1}-m+\deg G\,  .
 \end{eqnarray*}
  On the other hand the chain rule gives
  $$
    \jj_{x_{i_1},\cdots,x_{i_m}}(f_1,\cdots,f_{m-1},G)=
    \jj_{x_{i_1},\cdots,x_{i_m}}(f_1,\cdots,f_{m-1},f_m)\frac{\partial G}{\partial f_m}.
  $$
  Hence we get
  \begin{eqnarray*}
    & \deg \jj_{x_{i_1},\cdots,x_{i_m}}(f_1,\cdots,f_{m-1},f_m)+\deg\frac{\partial G}{\partial f_m} \leq
    d_1+\cdots+d_{m-1}-m+\deg G &\\
    &\deg\frac{\partial G}{\partial f_m}+d_m-(d_1+\cdots+d_{m-1}+d_m-m)+\deg
    \jj_{x_{i_1},\cdots,x_{i_m}}(f_1,\cdots,f_m)\leq \deg G\; . &
  \end{eqnarray*}
  In particular, when the maximum is realized,
  $$
    \deg\frac{\partial G}{\partial f_m}+d_m-(d_1+\cdots+d_{m-1}+d_m-m)+\max_{1\leq i_1,\cdots,i_m\leq n}\deg
    \jj_{x_{i_1},\cdots,x_{i_m}}(f_1,\cdots,f_m)
    \leq \deg G
  $$
  $$
    \deg\frac{\partial G}{\partial f_m}+d_m-(d_1+\cdots+d_{m-1}+d_m-m-\max_{1\leq i_1,\cdots,i_m\leq}\deg
    \jj_{x_{i_1},\cdots,x_{i_m}}(f_1,\cdots,f_m))\leq \deg G
  $$
  $$
    \deg\frac{\partial G}{\partial f_m}+d_m-\nabla\leq \deg G.
  $$
\end{proof}

In order to state our main theorem one needs to fix some more notations: we denote $\bar p$ the leading term of a polynomial $p\in K[x_1,\cdots,x_n]$ and for any subalgebra $A \subset K[x_1,\cdots,x_n]$, we denote $\gr(A):=K[\bar A]$ the subalgebra generated by $\bar A=\{\bar a|a\in A\}$. We define
$s_i$, $\forall 1\leq i\leq m$,  as the degree of the minimal, if
any, polynomial of $\bar f_i$ over $\Frac(\gr({K[f_j]}_{j\neq i}))$,
the field of fractions of the subalgebra generated by $\overline{K[f_j]}_{j\neq
i}=\overline{K[f_1,\cdots,f_{i-1},f_{i+1},\cdots,f_m]}$ and as
$+\infty$ otherwise. We denote $\lfloor\alpha\rfloor$ the integral part of a real number $\alpha$ and agree that $k/\infty=0$ when $0\leq k<\infty$.

\begin{thm}\label{main}
  Let $G$ be a polynomial in $K[f_1,\cdots,f_m]$. Then  the following minoration holds, $\forall
  1\leq i\leq m$,
  $$
    \deg G\geq d_i\cdot\deg_{f_i}G -\nabla\cdot\lfloor\frac{\deg_{f_i}G}{s_i}\rfloor.
  $$
\end{thm}


\begin{proof}
  It is of course sufficient to prove  it  for $i=m$.
  First remark that a polynomial
  $G=\sum g_i f_m^i\in A[f_m]$, where $g_i\in A:=K[f_1,\cdots,f_{m-1}]$, has degree strictly
  smaller than $\max_i \deg g_i+i\cdot d_m$ if and only if
  $$
    \widehat{G}:=\sum_{\deg g_i+i\cdot d_m=\max} \bar g_i \bar f_m^i=0
  $$
  so if $s_m=+\infty$, which means such an annihilation cannot
  occur, then the minoration in the Theorem is clear. Let's assume now
  that $\bar f_m$ does have a minimal polynomial $p(\bar f_m)=0$
  with
  $ p=p(X)\in F[X]$ where $F$ is the field of fractions of $\gr(A)$ and $X$ a new indeterminate (whence $s_m:=\deg_Xp$). The
  following easy lemma constitutes  the very improvement with respect to \cite{SU}: it simplifies the proof a lot, makes it more general and even  stronger in the sense that one does not need the estimate (\ref{est}) anymore.
  \begin{lem}\label{l1}
    Let $G=\sum g_if_m^i$ be in $A[f_m]$ and
    $$
      h(X):=\sum_{\deg g_i+i\cdot d_m=\max} \bar g_i  X^i\in \gr(A)[X]\;\;\; (\mbox{ hence }\widehat{G}=h(\bar f_m)).
    $$ If
    $\deg G<\deg_{f_m}G\cdot d_m$ then $\widehat{G}=0$ or,
    equivalently, $h(X)\in (p(X)):=p(X)\cdot F[X]$.\\
    Moreover if $h'(X)\neq 0$, where $h'$ is the derivative of $h$, then  $\widehat{\frac{\partial G}{\partial
    f_m}}=h'(\bar f_m)$
    and more generally, while  $h^{(k)}\neq
    0$, one has
    $\widehat{\frac{\partial^k G}{\partial f_m^k}}=h^{(k)}(\bar f_m)$.
  \end{lem}
  \begin{proof}
    If $\deg G<\deg_{f_m}G\cdot d_m$ then $\deg G<\max_i \deg g_i+i\cdot d_m$ and $\widehat G=0$ as already remarked
    above.\\
    Assume that $h'\neq 0$. One has
    $$
      \begin{array}{rlcl}
        \widehat G &=\sum_{i\in I} \bar g_i \bar f_m^i=h(\bar f_m) &\mbox{ where } & I:=\{i|\deg g_i+i\cdot d_m\geq \deg
        g_j+j\cdot d_m\,\forall j\}\mbox{ and}\\
        \widehat{\frac{\partial G}{\partial f_m}} & =\sum_{i\in I'} i\bar g_i \bar f_m^{i-1}&\mbox{ where }&
         I':=\{i|\deg ig_i+(i-1)\cdot d_m\geq \deg jg_j+(j-1)\cdot d_m\,\forall j\}.
       \end{array}
    $$
    It remains to notice  that $I'=I\cap \N^*$ when this intersection is not empty, which occurs exactly when $h'\neq
    0$.
  \end{proof}
  Let now $k$ be the maximal number such that $h(X)\in (p(X)^k)$. Clearly
  $\deg_{f_m}G\geq\deg h\geq k\cdot \deg p=ks_m$ hence $k\leq\lfloor\frac{\deg_{f_m}G}{s_m}\rfloor$. One has
  $h^{(k)}\notin (p(X))$ hence, by the Lemma,
  $$
    \deg\frac{\partial^k G}{\partial f_m^k}\geq d_m\cdot\deg_{f_m}\frac{\partial^k G}{\partial
    f_m^k}=d_m\cdot(\deg_{f_m}G-k)
  $$
  and, by property (\ref{para}),
  $$
    \deg G\geq d_m\cdot(\deg_{f_m}G-k)+k\cdot d_m-k\cdot \nabla=d_m\cdot\deg_{f_m}G-k\cdot \nabla.
  $$
\end{proof}
A straightforward computation gives the following
\begin{cor}\label{cN}
  Define, $\forall i=1,\cdots,m$, $N_i=N_i(f_1,\cdots,f_m):=s_id_i-\nabla$. Let $G$ be a polynomial in $K[f_1,\cdots,f_m]$
  and, $\forall i=1,\cdots,m$, let $\deg_{f_i}G=q_is_i+r_i$ be the
  euclidean division of $\deg_{f_i}G$ by $s_i$. Then the following
  minoration holds
  $$
     \deg G\geq q_i\cdot N_i +r_id_i.
  $$
\end{cor}
The special case $m=2$ corresponds to the main result of \cite{SU}
(where $s_1,s_2$ are easy to compute):
\begin{cor}\label{cSU}
  If $m=2$, $\sigma_i:=\frac{d_j}{\gcd(d_1,d_2)}$ with $(i,j)=(1,2)$ and $(2,1)$ and $N:=\sigma_1d_1-\nabla=
  \sigma_2d_2-\nabla$
   then the following minoration holds, for $i=1,\,2$,
  $$
    \deg G\geq q_i\cdot N +r_id_i
  $$
  where $\deg_{f_i}G=q_is_i+r_i$.
\end{cor}
\begin{proof} Let us prove it for $i=2$. By corollary \ref{cN} it suffices to prove that $s_2\geq\sigma_2$ \footnote{Actually equality holds, as proved in \cite{SU}  using Zaks Lemma, it is however possible to show it easily and without this result.}: $s_2$ is the degree of the minimal polynomial of $f_2$ over $\Frac(\gr(K[f_1]))=\Frac(K[\bar f_1])=K(\bar f_1)$:
\begin{eqnarray}\label{s2}
  &p(\bar f_2)=\bar f_2^{s_2}+p_{s_2-1}(\bar f_1)\bar f_2^{s_2-1}+\cdots+p_1(\bar f_1)\bar f_2+p_0(\bar
  f_1)=0&
\end{eqnarray}
hence $\exists 0\leq i\neq j\leq s_2$ such that $\deg p_i(\bar f_1)\bar f_2^i=\deg p_j(\bar f_1)\bar f_2^j$. It follows that $i\cdot d_2\equiv jd_2 \mod d_1$ whence $d_1\mid (i-j)d_2$
 and $i-j\in\Z^*\frac{d_1}{\gcd(d_1,d_2)}$ which gives $s_2\geq |i-j|\geq \frac{d_1}{\gcd(d_1,d_2)}=\sigma_2$.
\end{proof}
\begin{cor}\label{dG1}
Let $G$ be a polynomial in $K[f_1,\cdots,f_m]$ such that $\deg G=1$.
Then, $\forall i=1,\cdots,m$, $\deg_{f_i}G=0$ or $d_i=1$ or
$N_i=s_id_i-\nabla\leq 1$.
\end{cor}
\begin{proof}
Otherwise, by corollary \ref{cN}, $\deg G=1\geq
q_iN_i+r_id_i\geq\min\{N_i,d_i\}\geq 2$, a contradiction.
\end{proof}
\begin{cor}\label{caut}
  Assume $m=n$ and $K[f_1,\cdots,f_n]=K[x_1,\cdots,x_n]$ i.e.
  $f_1,\cdots,f_n$ define an automorphism (well-known fact). Then
   $\forall i=1,\cdots,n$, $d_i=1$ or $s_id_i\leq d_1+\cdots+d_n-n+1$. In
  particular, if $d_{\max}\geq d_j,\forall j$, and $d_{\max}\geq 2$
  (i.e. the automorphism is not affine) then $s_{\max}\leq n-1$.
\end{cor}
\begin{proof}
One has $\nabla=\nabla(f_1,\cdots,f_n)=d_1+\cdots+d_n-n-\deg
  \jj_{x_{1},\cdots,x_{n}}(f_1,\cdots,f_n)=d_1+\cdots+d_n-n$.
  Moreover $\forall j=1,\cdots,n$, there exists $G_j\in
  K[f_1,\cdots,f_n]$ such that $x_j=G_j$ and, $\forall
  i=1,\cdots,n$, $\deg_{f_i}G_j\leq 1$ for at least one
  $j=1,\cdots,n$ otherwise
  $K[f_1,\cdots,f_{j-1},f_{j+1},\cdots,f_m]=K[x_1,\cdots,x_n]$ which
  is impossible. Whence, by corollary \ref{dG1}, $d_i=1$ or $s_id_i\leq\nabla+1=
  d_1+\cdots+d_n-n+1$. With $d_{\max}$ one gets
  $s_{\max}d_{\max}\leq d_1+\cdots+d_n-n+1\leq
  nd_{\max}-n+1\leq nd_{\max}-1$ ($n\geq 2$) and it follows that $s_{\max}\leq n-1$.
\end{proof}
\begin{cor}[Tameness Theorem in dimension two]
  Every automorphism of $K[x_1,x_2]$ is tame i.e. a product of
  affine and elementary ones. Recall that an automorphism $\tau:K[x_1,x_2]\rightarrow
  K[x_1,x_2]$ is called elementary when, up to exchanging $x_1$ and $x_2$, $\tau(x_1)=x_1+p(x_2)$ and
  $\tau(x_2)=x_2$ for some $p(X)\in K[X]$.
\end{cor}
\begin{proof}
  Let $\alpha:K[x_1,x_2]\rightarrow
  K[x_1,x_2]$ be an automorphism defined by $\alpha(x_i)=f_i$ for
  $i=1,2$. We prove the corollary by induction on $d_1+d_2=\deg
  f_1+\deg f_2$.\\
  If $d_1+d_2=2$ then $d_1=d_2=1$ and $\alpha$ is affine.\\
  Assume $d_1+d_2\geq 3$. Without loss of generality $d_1\leq d_2$
  and $d_2\geq 2$ whence, by corollary \ref{caut}, $s_2=1$ and the
  relation (\ref{s2}) in the proof of corollary \ref{cSU} becomes:
  $\bar f_2=p(\bar f_1)$ where $p(X)$ must be of the form
  $p(X)=p_{s_1}X^{s_1}\in K[X]$. Taking the elementary automorphism
  $\tau$ defined $\tau(x_1)=x_1$ and $\tau(x_2)=x_2-p(X)$ one has a
  new pair
  $f_1':=\alpha\tau(x_1)=\alpha(x_1)=f_1$ and
  $f_2':=\alpha\tau(x_2)=\alpha(x_2-p(X))=f_2-p(f_1)$ with degrees
  $d_1'=d_1$ and $d_2'< d_2$ hence $d_1'+d_2'<d_1+d_2$. By induction
  $\alpha\tau$ is tame and so is $\alpha$.
\end{proof}

\vspace{1em}
\begin{tabular}{lll}
St\'ephane V\'en\'ereau\\
Mathematisches Institut\\
Universit\"at Basel\\
Rheinsprung 21, CH-4051 Basel\\
Switzerland\\
stephane.venereau@unibas.ch\\
\end{tabular}

\begin{thebibliography}{01}
\bibitem{ML} L. Makar-Limanov, {\em  Locally nilpotent derivations, a new ring invariant and applications}, preprint. 
\bibitem{Yu}J.T. Yu, {\em On relations between Jacobians and minimal polynomials}, Linear Algebra Appl. {\bf 221}, 1995, 19--29. MR {\bf 96c:}14014
\bibitem{SU} Shestakov, Ivan P. and Umirbaev, Ualbai U., {\it Poisson brackets and two-generated subalgebras of rings of polynomials}, J. Amer. Math. Soc. {\bf 17}(1), 2004, 181--196 (electronic).
\end{thebibliography}
\end{document}